\documentclass[A4paper, 12pt]{article}

\usepackage{hyperref}
\usepackage[utf8]{inputenc}
\usepackage[OT2,T1]{fontenc}
\usepackage[english,french]{babel}
\usepackage[osf]{mathpazo}

\usepackage{color}
\definecolor{FOB}{rgb}{0.,0.,0.7}
\definecolor{FOR}{rgb}{0.65,0.,0.}
\definecolor{FOfond}{rgb}{0.98,0.98,0.91}

\pagecolor{FOfond}

\usepackage{sectsty}
\allsectionsfont{\color{FOR}}

\usepackage{amssymb}
\usepackage{authblk}

\DeclareSymbolFont{cyrletters}{OT2}{wncyr}{m}{n}
\DeclareMathSymbol{\Sha}{\mathalpha}{cyrletters}{"58}
\DeclareMathSymbol{\rde}{\mathalpha}{cyrletters}{"64}

\def\bull{\vrule height .9ex width .8ex depth -.1ex }

\newcounter{thenum}
\def\texttheo{\relax}
\newenvironment{theorem}{\medbreak\refstepcounter{thenum}
\noindent\textsc{Theorem} %
\thethenum. \texttheo ---  \it  }{\rm }
\newenvironment{e-proposition}{\medbreak\refstepcounter{thenum}
\noindent\textsc{Proposition} \thethenum. ---  \it  }{\rm }

\newenvironment{e-definition}{\medbreak\refstepcounter{thenum}
\noindent\textsc{Definition} \thethenum. ---  \it  }{\rm }
\newenvironment{lemma}{\medbreak\refstepcounter{thenum}\noindent{\it Lemma} %
\thethenum. ---  \it  }{\rm }

\newenvironment{e-rem}{\medbreak\refstepcounter{thenum}{}%
 \thethenum) }{}

\newenvironment{e-ex}{\medbreak\refstepcounter{thenum}{}%
 \thethenum) }{}
\newenvironment{proof}{\smallbreak\noindent{\sc Proof.} --- \rm}{\quad\bull\smallskip\rm}

\begin{document}

\begin{center}{\LARGE\parindent=0pt{\color{FOR}
A  simple and constructive proof to a  generalization of L\"uroth's theorem

}}
\end{center}
\vskip2cm

\hbox to \hsize{\parindent =0pt\hbox to 2.5cm{\hfill}\hss\vbox{\hsize = 7cm {\large François \textsc{Ollivier}}
\bigskip

LIX, UMR CNRS 7161 

École polytechnique 

91128 Palaiseau \textsc{cedex}

France

\smallskip

{\tiny francois.ollivier@lix.polytechnique.fr}
} \hss\hss
\vbox{\hsize = 7cm {\large Brahim \textsc{Sadik}}
\bigskip

Département de Mathématiques 

Faculté des Sciences Semlalia 

B.P.\ 2390, 40000 Marrakech 

Maroc
\smallskip

{\tiny sadik@ucam.ac.ma}}\hss}
\vskip 0.3cm

\begin{center}\parindent =0pt March 2022
\end{center}
\vfill

{\small
\hbox to \hsize{\hss\vbox{\hsize= 6.2cm\selectlanguage{english}
\noindent \textbf{Abstract.} 
A generalization of L\"uroth's theorem expresses  that
  every transcendence degree $1$ subfield of the rational function
  field is  a  simple extension. In this note we show that a classical
  proof of this theorem also holds to prove this
  generalization.\hfill\break
  \phantom{toto}
  \smallskip

}
\hss\hss

\vbox{\hsize= 6.3cm\selectlanguage{french}
\noindent \textbf{Résumé.} 
Une généralisation du théorème de Lüroth affirme que tout sous-corps
de degré de transcendance $1$ d'un corps de fractions rationnelles est
une extension simple. Dans cette note, nous montrons qu'une preuve
classique permet également de prouver cette généralisation.
  \smallskip

  \noindent Mots-clés: Th.\ de Lüroth,
  degré de transcendance $1$, 
  extension simple.
}\hss}}
\medskip

\noindent{\tiny Authors' extended version of:
  \textsc{Ollivier} (François) and \textsc{Sadik} (Brahim), ``A  simple and constructive proof to a  generalization of Lüroth's theorem'',
  \textit{Turkish Journal of Mathematics}, on line, waiting for inclusion in an
  issue,  2022.
 \hfill
 DOI:~\href{https://journals.tubitak.gov.tr/math/inpress.htm}{10.3906/mat-2110-11}\break

} 

\vfill\eject
\selectlanguage{english}

\section*{Introduction}
\label{Sec:1}

L\"uroth's theorem (\cite{luroth}) plays an important role in the
theory of rational curves. A generalization of this theorem to
transcendence degree $1$ subfields of rational functions field was
proven by Igusa in \cite{igusa}. A purely field theoretic proof of
this generalization was given by Samuel in \cite{samuel}. In this note
we give a simple and constructive proof of this result, based on a
classical proof \cite[10.2~p.218]{waerden}.\\ Let $k$ be a field and
$k(x)$ be the rational functions field in $n$ variables
$x_1,\ldots,x_n$. Let $\cal K$ be a field extension of $k$ that is a
subfield of $k(x)$. To the subfield $\cal K$ we associate the prime
ideal ${\Delta({\cal K})}$ which consists of all polynomials of ${\cal
  K}[y_1,\ldots,y_n]$ that vanish for $y_1=x_1,\ldots,y_n=x_n$. When
the subfield $\cal K$ has transcendence degree $1$ over $k$, the
associated ideal is principal. The idea of our proof relies on a
simple relation between coefficients of a generator of the associated
ideal ${\Delta({\cal K})}$ and a generator of the subfield $\cal
K$. When $\cal K$ is finitely generated, we can compute a rational
fraction $v$ in $k(x)$ such that ${\cal K} = k(v)$.  For this, we use
some methods developped by the first author in \cite{ollivier} to get
a generator of ${\Delta({\cal K})}$ by computing a Gr\"obner basis or
a characteristic set.

\section*{Main result}
\label{Sec:2}
Let $k$ be a field and $x_1,\ldots,x_n,y_1,\ldots,y_n$ be $2n$
indeterminates over $k$. We use the notations $x$ for
$x_1,\ldots,x_n$ and $y$ for $y_1,\ldots,y_n$.  
If  $\cal K$ is a field extension of $k$ in $k(x)$ we define the ideal
${\Delta({\cal K})}$ to be the prime ideal of all polynomials in
${\cal K}[y]$ that vanish for $y_1=x_1,\ldots,y_n=x_n$.  
$${\Delta({\cal K})} = \{P\in {\cal K}[y]\,\,\, : \,\,\,
P(x_1,\ldots,x_n) = 0\}.$$ 

\begin{lemma}\label{delta}
Let $\cal K$ be a field extension of $k$ in $k(x)$ with transcendence
degree $1$ over $k$.

i) The ideal $\Delta({\cal K})$ is principal in
${\cal K}[y]$.

ii) If ${\cal K}_{1}\subset {\cal K}_{2}$ and $\Delta(K_{i})={\cal
  K}_{i}[y]G$, for $i=1,2$, then ${\cal K}_{1}={\cal K}_{2}$.

iii) $\Delta({\cal K})=\tilde\Delta({\cal
  K}):=(p(y)-p(x)/q(x)q(y)|p/q\in{\cal K})$. 

iv) The ideal $\hat\Delta({\cal K}):=k[x]\Delta({\cal K})\cap k[x,y]$
is a radical ideal, which is equal to $(q(x)p(y)-p(x)q(y)|p/q\in{\cal
  K})$.

v) Let $G$ be such that $\Delta({\cal K})=(G)$, with
$G=\sum_{j=0}^{d}p_{j}(x)/q_{j}(x)y^{j}$ and $GCD(p_{j},q_{j})=1$, for
$0\le j\le d$. Let $Q:=PPCM(q_{j}\>|\>0\le j\le d)$, then $\hat G:=QG$ is
such that $G(y,x)=-G(x,y)$ and $\deg_{x}\hat G=\deg_{y}\hat G=d$. 
\end{lemma}
\begin{proof}
i) In the  unique factorization domain ${\cal K}[y]$, the prime ideal
$\Delta({\cal K})$ has codimension $1$. Hence, it is   principal. 

ii) Assume that ${\cal K}_{1}\neq {\cal K}_{2}$. There exists
$p(x)/q(x)\in {\cal K}_{2}$ a reduced fraction, with $p(x)/q(x)\notin
{\cal K}_{1}$. The set $\{1, p(x)/q(x)\}$ may be completed to form a
basis $\{e_{1}=1, e_{2}=p/q, \ldots, e_{s}\}$ of ${\cal K}_{2}$ as a
${\cal K}_{1}$-vector space. Then, $e$ is also a basis of ${\cal
  K}_{2}[y]={\cal K}_{2}{\cal K}_{1}[y]$ as a ${\cal K}_{1}[y]$-module
and $Ge$ is a basis of $\Delta({\cal K}_{2})={\cal K}_{2}\Delta({\cal
  K}_{1})$ as a ${\cal K}_{1}[y]$-module. So, $p(y)-p(x)/q(x)q(y)\in
\Delta({\cal K}_{2})$ is equal to $p(y)e_{1}-q(y)e_{2}$, which implies
that $G$ divides $p$ and $q$, a contradiction.

iii) We remark that $\tilde\Delta({\cal K})$ does not define any prime
component containing polynomials $k[y]$, so that
$\tilde\Delta({\cal K}):k[y]=\tilde\Delta({\cal K})$. The
inclusion $\supset$ is
immediate. Let $P\in\Delta({\cal K})$ with
$P(x,y)=\sum_{j=0}^{s}p_{j}(x)/q_{j}(x)y^{j}$. We have $P(x,x)=0$ and by
symmetry $P(y,y)=0$, so
$P=P(x,y)-P(y,y)=\sum_{j=0}^{s}(p_{j}(x)/q_{j}(x)-p_{j}(y)/q_{j}(y))y^{j}$.
So, throwing away denominators in $k[y]$,
$\prod_{j=1}^{s}q_{i}(y)P\in\tilde\Delta({\cal K})$, so that
$P\in\tilde\Delta({\cal K}):k[y]=\tilde\Delta({\cal K})$,
hence the result.

iv) The ideal $\Delta({\cal K})$ is prime, so that $k(x)\Delta({\cal
  K})$ and $\hat\Delta({\cal K})$ are radical. We remark that $\hat\Delta({\cal K})$ does not define any prime
component containing polynomials $k[x]$ or in $k[y]$, so that
$\hat\Delta({\cal K}):(k[x]k[y])=\hat\Delta({\cal K})$. The inclusion
$\supset$ is immediate. Using the generators $p(y)-p(x)/q(x)q(y)$,
$p/q\in{\cal K}$, a finite set of fractions $\Sigma$ is enough by
Noetherianity, so that $\prod_{p/q\in\Sigma}q(x)\delta({\cal
  K})\subset (p(y)-p(x)/q(x)q(y)|p/q\in{\cal K})$, which provides the
reverse inclusion, using the previous remark.

v) By construction, $\hat G$ is a generator of $\hat\Delta({\cal
  K})$. All the generators of $\hat\Delta({\cal
  K})$ in~iv) being antisymmetric, $\hat G$ is antysymmetric, which
also implies that $\deg_{x}\hat G=\deg_{y}\hat G=d$.
\end{proof}

\begin{theorem}
Let $\cal K$ be a field extension of $k$ in $k(x)$ with transcendence
degree $1$ over $k$. Then, there exists $v$ in $k(x)$ such that ${\cal
  K} = k(v)$. 
\end{theorem}
\begin{proof}
By lem.~\ref{delta}~i), the prime ideal $\Delta({\cal K})$ of ${\cal
  K}[y]$ is principal. Let $G$ be a monic polynomial such that
$\Delta({\cal K}) = (G)$ in ${\cal K}[y]$. Let $c_0(x),\ldots,c_r(x)$
be the coefficients of $F$ as a polynomial in ${\cal K}[y]$. Since
$x_1,\ldots,x_n$ are transcendental over $k$ there must be a
coefficient $v:=c_{i}$ that lies in ${\cal K}\backslash k$.

Write $v=\frac{f(x)}{g(x)}$ where $f$ and $g$ are relatively prime in
$k[x]$. By lem.~\ref{delta}~v), $\max(\deg_{x} f,\deg_{x} g)\le
d:=\deg_{x} G$. As $g(x)f(y)-f(x)g(y)$ is a multiple of $\hat G$,
$\max(\deg_{x} f,\deg_{x} g)=d$. Let $D := f(y)-vg(y)$. As
$D\in\Delta({\cal K})$, the remainder of the Euclidean division of $G$
by $D$ is also in $\Delta({\cal K})$ and of degree less than the
degree of $G$. It must then be $0$. Therefore $D$ is a generator of
$\Delta({k(v)})$ and of $\Delta({\cal K})$, with $k(v)\subset{\cal
  K}$, and by lem.~\ref{delta}~ii), we need have ${\Delta}({\cal K}) =
\Delta(k(v))$ and ${\cal K} = k(v)$.
 \end{proof}
 
The following result, given by the first author in \cite[prop.~4
  p.~35]{ollivier} and \cite[th.~1]{ollivier2} in a differential
setting that includes the algebraic case, permits to compute a basis
for the ideal ${\Delta}({\cal K})$.
\begin{e-proposition}
Let ${\cal K} = k(f_1,\ldots,f_r)$ where the $f_i=\frac{P_i}{Q_i}$ are
elements of $k(x)$. Let $u$ be a new indeterminate and consider the
ideal   
$${\cal J} =
\left(P_1(y)-f_1Q_1(y),\ldots,P_r(y)-f_rQ_r(y),u\left(\prod_{i=1}^r
Q_i(y) - 1\right)\right)$$  in ${\cal K}[y,u]$. Then \\ 
$${\Delta}({\cal K}) = {\cal J}\cap {\cal K}[y].$$
\end{e-proposition}

\section*{Conclusion}
\label{Sec:3}
 A generalization of L\"uroth's theorem to differential algebra has
 been proven by J. Ritt in \cite{ritt}. One can use the theory of
 characteristic sets to compute a generator of a finitely generated
 differential subfield of the differential field ${\cal F}\langle
 y\rangle$ where $\cal F$ is an ordinary differential field and $y$ is
 a differential indeterminate. In a forthcoming work we will show that
 L\"uroth's theorem can be generalized to one differential
 transcendence degree subfields of the differential field ${\cal
   F}\langle y_1,\ldots,y_n\rangle$.


\begin{thebibliography}{9}
\bibitem{igusa} \textsc{Igusa} (Jun-ichi), ``On a theorem of
  Lueroth'', \textit{Memoirs of the College of Science}, Univ.\ of
  Kyoto, Series~A, vol.~26, Math. n$^{\rm o}$~3, 251--253, 1951.
  
\bibitem{luroth} \textsc{Lüroth} (Jacob), ``Beweis eines Satzes \"uber
  rationale Curven'', \textit{Mathematische Annalen}~9, 163--165,
  1875.
  
\bibitem{ollivier} \textsc{Ollivier} (François), \textit{Le problème
    d'identifiabilité structurelle globale : approche théorique,
    méthodes effectives et bornes de complexité}, Thèse de
  doctorat en science, École polytechnique, 1991.
  
\bibitem{ollivier2} \textsc{Ollivier} (François), ``Standard bases of
  differential ideals'', proceedings of AAECC 1990, \textit{Lecture
    Notes in Computer Science}, vol.~508, Springer, Berlin,
  Heidelberg, 304--321, 1990.
  
\bibitem{ritt} \textsc{Ritt} (Joseph Fels), \textit{Differential Algebra},
  Amer.\ Math.\ Soc.\ Colloquium Publication, vol.~33, 
  Providence, 1950.
  
\bibitem{samuel} \textsc{Samuel} (Pierre), ``Some Remarks on Lüroth's Theorem'',
  \textit{Memoirs of the College of Science}, Univ.\ of
  Kyoto, Series~A, vol.~27, Math. n$^{\rm o}$~3, 223--224, 1953.

\bibitem{waerden} \textsc{Van der Waerden} (Bartel Leendert),
  \textit{Algebra}, vol.~1, Frederick Ungar Publishing Company, New
  York, 1970, reprint by Springer 1991.
  
 \end{thebibliography}
\end{document}